\documentclass[12pt]{amsart}
\usepackage{geometry}
\usepackage{amsmath}
\usepackage{amssymb}
\usepackage{latexsym}
\usepackage{mathrsfs}
\usepackage{enumitem}
\usepackage[all]{xy}

\allowdisplaybreaks
\theoremstyle{lemma}
\newtheorem{theorem}{Theorem}[section]
\newtheorem{lemma}[theorem]{Lemma}

\theoremstyle{definition}
\newtheorem{definition}[theorem]{Definition}

\newtheorem{example}[theorem]{Example}

\theoremstyle{proposition}
\newtheorem{proposition}[theorem]{Proposition}

\theoremstyle{corollary}
\newtheorem{corollary}[theorem]{Corollary}

\theoremstyle{remark}

\theoremstyle{notation}

\theoremstyle{claim}

\numberwithin{equation}{section}

\newcommand{\defeq}{\mathrel{\mathop:}=}

\begin{document}

\title[Comparison and pure infiniteness of crossed products ]{Comparison and pure infiniteness of crossed products }

\author{Xin Ma}
\email{dongodel@math.tamu.edu}
\address{Department of Mathematics,
         Texas A\&M University,
         Collage Station, TX 77843}


\date{June 17, 2019.}


\begin{abstract}
Let $\alpha: G\curvearrowright X$ be a continuous action of an infinite countable group on a compact
Hausdorff space. We show
that, under the hypothesis that the action $\alpha$ is topologically free and has no $G$-invariant regular Borel probability measure on $X$, dynamical
comparison implies that the reduced crossed product of $\alpha$ is purely infinite and
simple. This result, as an application, shows a dichotomy between stable finiteness and pure infiniteness for reduced crossed products arising from actions satisfying dynamical comparison. We also introduce the concepts of paradoxical comparison and the uniform tower property. Under the hypothesis that
the action $\alpha$ is exact and essentially free, we show that paradoxical comparison together with the uniform tower property implies that the reduced crossed product of
$\alpha$ is purely infinite. As applications, we provide new results on pure infiniteness of reduced crossed products in which the underlying spaces are not necessarily zero-dimensional. Finally, we study the type semigroups of
actions on the Cantor set in order to establish the equivalence of almost unperforation of the type semigroup and comparison. This sheds a light to a question arising in the paper of R{\o}rdam and Sierakowski \cite{R-S}.
\end{abstract}

\maketitle

\section{Introduction}
Reduced crossed products of the form $C(X)\rtimes_r G$ arising from topological dynamical systems, say from $(X, G,
\alpha)$ for a countable discrete group $G$, an infinite compact Hausdorff space $X$ and a continuous action
$\alpha$, have long been an important source of examples and motivation for the study of $C^\ast$-algebras. A special class of $C^\ast$-algebras, which consists of all purely infinite simple separable and nuclear $C^\ast$-algebras, called Kirchberg algebras, are of particular interest because of their classification by K- or KK-theory obtained by Kirchberg and Phillips in the mid 1990s. Motivated by the classification of Kirchberg algebras, it is important to see what properties of dynamics imply that reduced crossed products are Kirchberg algebras and it is also of independent interest to see what dynamical properties imply that reduced crossed products are purely infinite.

It is well known that if the action $G\curvearrowright X$ is topologically free and minimal then the reduced crossed product $C(X)\rtimes_r G$ is simple (see \cite{A-S}) and it is also known that the crossed product $C(X)\rtimes_r G$ is nuclear if and only if the action $G\curvearrowright X$ is amenable (see \cite{B-Ozawa}). Archbold and Spielberg \cite{A-S} showed that $C(X)\rtimes G$ is simple if and only if the action is minimal, topologically free and \textit{regular} (meaning that the reduced crossed product coincides with the full crossed product). These imply that $C(X)\rtimes_r G$ is simple and nuclear if and only if the action is minimal, topologically free and amenable. See more details in the introduction of \cite{R-S}.

So far there have been several sufficient dynamical conditions that imply the pure infiniteness of reduced crossed products provided that the action $G\curvearrowright X$ is topologically free. Laca and Spielberg
\cite{L-S} showed that the reduced crossed product $C(X)\rtimes_r G$ is purely infinite provided that the action
$G\curvearrowright X$ is also a strong boundary action, which means that $X$ is infinite and that any two non-empty
open subsets of $X$ can be translated by group elements to cover the entire space $X$. Jolissaint and Robertson \cite{J-R}
generalized this result and showed that it is sufficient to require that the action is $n$-filling, which means
the entire space can be covered by translations of $n$ open subsets instead of two open subsets of $X$. Note
that if an action $G\curvearrowright X$ is either a strong boundary action or $n$-filling then
it is minimal.

In this paper, we continue the research of pure infiniteness on reduced crossed products. It is worth observing
that the phenomenon of paradoxicality is essential in the study of pure infiniteness. In addition, since a purely infinite reduced crossed product is necessarily traceless, there is no $G$-invariant regular Borel probability measure on the space. In this situation dynamical comparison, which is a well-known
dynamical analogue of strict comparison in the $C^\ast$-setting, also has paradoxical flavour and therefore is a good candidate to imply pure infiniteness of reduced crossed products. Indeed, if there is no $G$-invariant regular Borel probability measure on the space, we will show that dynamical comparison, which is weaker than $n$-filling, still implies that the reduced crossed product is purely infinite. We also remark that dynamical comparison also implies minimality of the action when there is no $G$-invariant regular Borel probability measure. The following theorem is our first main result.

\begin{theorem}
Let $G$ be a countable discrete infinite group, $X$ a compact Hausdorff space and $\alpha: G\curvearrowright X$ a minimal topologically free continuous action of $G$ on $X$.  Suppose that there is no $G$-invariant regular Borel probability measure on $X$ and $\alpha$ has dynamical comparison. Then the reduced crossed product $C(X)\rtimes_r G$ arising from $\alpha$ is purely infinite and simple.
\end{theorem}

An application of this theorem is the following dichotomy result. The $C^\ast$-enthusiast of old used to hope
that the trace/traceless may determine a dichotomy between stably finite and purely infinite unital simple separable and nuclear $C^\ast$-algebras.
However, R{\o}rdam \cite{Ror} disproved this conjecture by providing an example of a unital simple separable
nuclear $C^\ast$-algebra containing both an infinite and a non-zero finite projection. Nevertheless, if we
restrict to reduced crossed products of actions with the regularity property of dynamical
comparison then the conjecture is true. In fact the dichotomy holds even if the reduced crossed product is neither nuclear nor separable. Indeed, suppose that $\alpha: G\curvearrowright X$ is a minimal  and
topologically free action. Every tracial states on $C(X)\rtimes_r G$ induces a $G$-invariant regular Borel probability measure on $X$ when it restrict to $C(X)$. On the other hand, suppose that $\mu$ is a $G$-invariant regular Borel probability measure on $X$. It induces a faithful tracial state $\tau$ on the reduced crossed product $C(X)\rtimes_r G$ defined by
$\tau(a)=\int_X E(a)\, d\mu$, where $E$ is the canonical faithful conditional expectation from $C(X)\rtimes_r G$
onto $C(X)$. In this case it is well-known that $C(X)\rtimes_r G$ is stably finite. Combining this fact with
Theorem 1.1, we obtain the following dichotomy.

\begin{corollary}
Let $G$ be a countable discrete group, $X$ an infinite compact Hausdorff space and $\alpha: G\curvearrowright X$ a minimal topologically free continuous action of $G$ on $X$. Suppose that the action $\alpha$ has dynamical comparison.  Then the reduced crossed product $C(X)\rtimes_r G$ is simple and is either stably finite or purely infinite.
\end{corollary}

To see pure infiniteness of reduced crossed products induced by non-minimal actions R{\o}rdam and Sierakowski \cite{R-S} first introduced $(G, \tau_X)$-paradoxicality,
where $\tau_X$ is the topology on $X$. A subset $Y$ of $X$ is called  $(G, \tau_X)$-paradoxical if there are two
finite open covers of $Y$ such that there are group elements which translate all sets appearing as a member of one of these two
covers to pairwise disjoint open subsets of $Y$.  It was mentioned in \cite{R-S} that if an action $G\curvearrowright X$ is $n$-filling then every non-empty open subset of $X$ is $(G, \tau_X)$-paradoxical and the action is minimal. Based on this notion,
R{\o}rdam and Sierakowski \cite{R-S} showed that, under the hypothesis that the group $G$ is exact and the action $G\curvearrowright X$ is essentially free,  if $X$ has a basis of clopen $(G,
\tau_X)$-paradoxical sets then $C(X)\rtimes_r G$ is purely infinite. In this present paper for an action which is not necessarily minimal, we will introduce a new concept called paradoxical comparison (Definition 4.3 below), which is an analogue of dynamical comparison for non-necessarily minimal actions which has no $G$-invariant regular Borel probability measures since we will show below that they are equivalent when the action is minimal. In addition, if the underlying space $X$ is zero-dimensional then our paradoxical comparison is equivalent to the statement that every clopen subset of $X$ is $(G, \tau_X)$-paradoxical. Therefore, motivated by the result of R{\o}rdam and Sierakowski \cite{R-S} mentioned above, our paradoxical comparison is a good candidate to show pure infiniteness for reduced crossed products induced by non-minimal actions on a compact space with higher dimension. Indeed, our paradoxical comparison implies pure
infiniteness of a reduced crossed product if its action is exact, essentially free and have an additional property called the uniform tower property (Definition 4.7 below).

One advantage of considering
dynamical comparison and paradoxical comparison is that they allow us to unify all of the above known sufficient criteria for pure infiniteness into one framework. The following is the second main result in this paper.

\begin{theorem}
Let $G$ be a countable infinite discrete group, $X$ a compact Hausdorff space and $\alpha: G\curvearrowright X$ an exact essentially free continuous action of $G$ on $X$. Suppose that the action $\alpha$ has paradoxical comparison as well as the uniform tower property. Then the reduced crossed product $C(X)\rtimes_r G$ arising from $\alpha$ is purely infinite.
\end{theorem}

We will see below that a minimal action $\alpha: G\curvearrowright X$ in the assumption of Theorem 1.1 trivially satisfies the assumption in Theorem 1.3. Therefore, Theorem 1.3 is a generalization of Theorem 1.1 for actions which are not necessarily minimal.  As an application of Theorem 1.3 we have the following results. The first one is a direct application while the second one needs some additional work.

\begin{corollary}
Let $G$ be a countable infinite discrete group, $X$ a compact Hausdorff space and $\alpha: G\curvearrowright X$ an exact essentially free continuous action of $G$ on $X$. Suppose that the action $\alpha$ has paradoxical comparison and there are only finitely many $G$-invariant closed subsets of $X$. Then the reduced crossed product $C(X)\rtimes_r G$ arising from $\alpha$ is purely infinite and has finitely many ideals.
\end{corollary}

In particular if an action $\alpha: G\curvearrowright X$ decomposes into finitely many minimal subsystems then the crossed product is purely infinite.  We also have the following result for ``amplifications'' of minimal topologically free actions, i.e., products of such an action with a trivial action. Indeed, the space $Y$ in the corollary below may be viewed as an index set so that  $\alpha:
G\curvearrowright X\times Y$ decomposes into $|Y|$-many disjoint copies of minimal subsystems of $\beta:
G\curvearrowright X$.

\begin{corollary}
Let $G$ be a countable infinite exact discrete group, $X$ a compact Hausdorff space and $\beta: G\curvearrowright X$ a minimal topologically free continuous action of $G$ on $X$.  Suppose that there is no $G$-invariant regular Borel probability measure on $X$ and $\beta$ has dynamical comparison. Let $Y$ be another compact Hausdorff space. Let $\alpha: G\curvearrowright X\times Y$ be an action defined by $\alpha_g((x, y))=(\beta_g(x),y)$. Then $C(X\times Y)\rtimes_{\alpha,r}G$ is purely infinite.
\end{corollary}

We remark that Corollary 1.5 is still true even $G$ is not exact by combining Theorem 1.1 and an algebraic argument.  To see this, we first observe that $C(X\times Y)\rtimes_{\alpha,r}G\cong C(Y)\otimes
C(X)\rtimes_{\beta,r}G$.  Theorem 1.1 above shows that $C(X)\rtimes_{\beta,r}G$ is purely infinite and simple and thus is of real rank zero. Then Corollary 6.9 in \cite{K-Ror} implies that $C(X)\rtimes_{\beta,r}G$ is strongly purely infinite.  Therefore, Theorem 1.3 in \cite{Kir-S} shows that $C(Y)\otimes
C(X)\rtimes_{\beta,r}G$ is strongly pure infinite since $C(Y)$ is nuclear. In light of this, our extra contribution is providing a pure dynamical approach to this interesting fact as an application of Theorem 1.3.

Dynamical comparison and paradoxical comparison also relate to the almost unperforation of the type semigroups of
actions on the Cantor set. It has been asked in \cite{TimR} and \cite{R-S} to what extent the type semigroup
of an action on Cantor set is almost unperforated. For a minimal free action $\alpha: G\curvearrowright X$ of an
amenable discrete infinite group $G$, Kerr \cite{D} showed that if the type semigroup of $\alpha$, denoted by $W(X, G)$, is almost unperforated then $\alpha$ has dynamical comparison. In addition, he showed that if the action $\alpha$ satisfies a notion called almost
finiteness then $W(X, G)$ is almost unperforated. A recent work of Kerr and Szab\'{o} \cite{D-G} showed that for
such an action $\alpha$ on the Cantor set $X$, it has dynamical comparison if and only if it is almost finite. Therefore $W(X, G)$ is almost unperforated if and only if
$\alpha: G\curvearrowright X$ has dynamical comparison provided that $G$ is amenable and the action $\alpha$ is
minimal and free. In this paper, we claim the same conclusion under the hypothesis that the action $\alpha$ is
minimal and has no $G$-invariant Borel probability measures on $X$ (see Proposition 6.2 below). In addition, this result is also covered by our third main result below for actions which are not necessarily minimal.

\begin{theorem}
Suppose  that $\alpha: G\curvearrowright X$ is an action on the Cantor space $X$ such that there is no $G$-invariant non-trivial Borel measure on $X$. Then the type semigroup $W(X, G)$ is almost unperforated if and only if the action has paradoxical comparison.
\end{theorem}

R{\o}rdam and Sierakowski \cite{R-S} asked that whether there is an example where the type semigroup is not almost unperforated and to what extent the type semigroup $W(X, G)$ is almost unperforated (or purely infinite). P. Ara and R. Exel \cite{A-E} constructed an action of a finitely generated free group on the Cantor set for which the type semigroup is not almost unperforated. Our Theorem 1.6 then sheds a light to the second part of R{\o}rdam and Sierakowski's question in the case that there is no $G$-invariant non-trivial Borel measure on the Cantor set $X$.  What we actually show in this case is that the action has paradoxical comparison if and only if the type semigroup $W(X, G)$ is almost unperforated if and only if the type semigroup $W(X, G)$ is purely infinite.

Our paper is organised as follows: in Section 2 we review the necessary concepts, definitions and preliminary results. In Section 3 we study dynamical comparison and prove Theorem 1.1.  In Section 4, we study actions that are not necessarily minimal by introducing paradoxical comparison and the uniform tower property. In addition, we will prove Theorem 1.3 there. In Section 5 we prove Corollary 1.4 and Corollary 1.5 by using Theorem 1.3. In Section 6 we focus on actions on the Cantor set and study the type semigroups of actions to establish Theorem 1.6.

\section{Preliminaries}
In this section, we recall some terminology and definitions used in the paper.
Throughout the paper $G$ denotes a countable infinite group, $X$ denotes an infinite compact Hausdorff
topological space and $\alpha: G\curvearrowright X$ denotes a continuous action of $G$ on $X$. We write $M(X)$ for the convex set of all regular Borel probability measures on $X$, which is a weak* compact subset of $C(X)^\ast$. We write $M_G(X)$ for the convex set of $G$-invariant regular Borel probability measures on $X$, which is a weak* compact subset of $M(X)$. A Borel measure $\mu$ (may be unbounded) on a zero-dimensional space $X$ is said to be \textit{non-trivial} if there is a clopen subset $O$ of $X$ such that $0<\mu(O)<\infty$.

Now, we recall the definitions of strong boundary action and the $n$-filling property.

\begin{definition}(\cite[Definition 1]{L-S})
An action $\alpha: G\curvearrowright X$ is said to be a \emph{strong boundary action} if for every closed subset $F$ and non-empty open subset $O$ of $X$ there exists a $g\in G$ such that $gF\subset O$.
\end{definition}

\begin{definition}(\cite[Proposition 0.3]{J-R})
An action $\alpha: G\curvearrowright X$ is called \emph{$n$-filling} if for any nonempty open subsets $U_1, U_2,..., U_n$ of $X$ there exist $g_1, g_2,..., g_n\in G$ such that $\bigcup_{i=1}^n g_iU_i=X$.
\end{definition}

It is not hard to see that if an action $\alpha: G\curvearrowright X$ is a strong boundary action then it is $2$-filling.  If an action $\alpha: G\curvearrowright X$ is $n$-filling for some $n\in \mathbb{N}$ then it is necessarily minimal and has no $G$-invariant regular Borel probability measure. In addition, the space $X$ has to be \emph{perfect} (meaning that the space has no isolated point). Dynamical comparison is a well-known dynamical analogue of strict comparison in $C^\ast$-setting. The idea dates back to Winter in 2012 and was discussed in \cite{B} and \cite{D}. We record here the version that appeared in \cite{D}.

\begin{definition}(\cite[Definition 3.1]{D})
Let $F$ be a closed subset and $O$ an open subset of $X$. We say $F$ subequivalent to $O$, denoting by $F\prec O$, if there exists a finite collection $\mathcal{U}$ of open subsets of $X$ which cover $F$, an $s_U\in G$ for each $U\in \mathcal{U}$ such that the images $s_UU$ for $U\in \mathcal{U}$ are pairwise disjoint subsets of $O$.  Let $V$ be another open subset of $X$. If $F\prec O$ for every closed subset $F\subset V$, we say $V$ subequivalent to $O$ and denote by $V\prec O$.
\end{definition}

\begin{definition}(\cite[Definition 3.2]{D})
An action $\alpha: G\curvearrowright X$ is said to have \emph{dynamical comparison} if $V\prec O$ for every open set $V\subset X$ and nonempty open set $O\subset X$ satisfying  $\mu(V)<\mu(O)$ for all $\mu\in M_G(X)$.
\end{definition}

Now suppose that $\alpha$ is an action such that there is no $G$-invariant regular Borel probability measure on
$X$. If $\alpha$ has dynamical comparison then every two nonempty open subsets of $X$ are subsequivalent to each
other in the sense of Definition 2.3. In addition, it can be verified that the action $\alpha$ has to be minimal
and the space $X$ has to be perfect. Indeed, for every $x\in X$ and non-empty open subset $O$ of $X$ there is a
group element $g\in G$ such that $g\{x\}\subset O$ since $\alpha$ has dynamical comparison. This verifies that
the action is minimal. In addition, it is not hard to see $|F|\leq |O|$ for every closed set $F$ and open set $O$ satisfying $F\prec O$ by Definition 2.3. Suppose that there is an open set whose cardinality is one. Observe that then any closed set containing exactly two points is subequivalent to this open set since $\alpha$ has dynamical
comparison, which is a contradiction to the cardinality inequality mentioned above.  This implies that the
cardinality of an open set cannot be one and thus the space is perfect.

It is not hard to see that if an action $\alpha$ is either $n$-filling or is a strong boundary action then it has dynamical comparison.  Indeed, suppose that the action $\alpha: G\curvearrowright X$ is $n$-filling. Then there is no $G$-invariant measure on $X$ and it suffices to show $V\prec O$ for two arbitrary
non-empty open sets $O, V$. For every closed set $F\subset V$, choose $n$ pairwise disjoint non-empty open subsets
$O_1, O_2,\dots,O_n$ of $O$ where all of these open sets contain more than one point. Since the space is Hausdorff and perfect, we can do this by choosing $n$ different points $x_1,x_2,\dots,x_n\in O$ and non-trivial open
neighbourhoods $O_i$ of $x_i$ for all $i=1,2,\cdots,n$ so that  $O_i\cap O_j=\emptyset$ whenever $i\neq j$ . Then
there are  $t_1, t_2,\dots, t_n\in G$ such that $\bigcup_{i=1}^n t_iO_i=X\supset F$, whence $\{t^{-1}_i: i=1,2,\dots, n\}$ and $ \{t_iO_i:i=1,2,\dots,n\}$ witness that $F\prec O$. Then one has $V\prec O$ because $F$ is an arbitrary closed subset of $V$. In particular, suppose now that $\alpha: G\curvearrowright X$ is a strong boundary action. It is $2$-filling and thus has dynamical comparison.

Recall that an action $\alpha: G\curvearrowright X$ is said to be \emph{essentially free} provided that, for every closed $G$-invariant subset $Y\subset X$, the subset of points in $Y$ with trivial isotropy, say $\{x\in Y: G_x=\{e\}\}$, is dense in $Y$, where $G_x=\{t\in G: tx=x\}$ (see \cite{Renault}). An action is said to be \textit{topologically free} provided that the set $\{x\in X: G_x=\{e\}\}$, is dense in $X$ and this is equivalent to that the fixed point set of each nontrivial element $t$ of $G$, $\{x\in X: tx=x\}$, is nowhere dense, i.e., the open interior of $\{x\in X: tx=x\}$ is empty. It is not hard to see that essentially freeness means that the restricted action to each $G$-invariant closed subspace  is topologically free with respect to the relative topology and thus these two concepts are equivalent when the action is minimal.

Now, we recall some notions of crossed product $C^\ast$-algebras and Cuntz comparison on (general) $C^\ast$-algebras. For general background on crossed product $C^\ast$-algebras we refer to \cite{B-Ozawa} and \cite{Williams}.   Let $A$ be a $C^\ast$-algebra on which there is a $G$-action. For every $G$-invariant ideal $I$ in $A$, the natural maps in the following short exact sequence:

\[(\ast)\ \ \ \ \xymatrix@C=0.5cm{
  0 \ar[r] & I \ar[rr]^{\iota} && A \ar[rr]^{\rho} && A/I \ar[r] & 0 }\]
extend canonically to maps at the level of reduced crossed products, giving rise to the possibly non-exact sequence
\[(\star)\ \ \ \ \xymatrix@C=0.5cm{
  0 \ar[r] & I\rtimes_r G \ar[rr]^{\iota\rtimes_r id} && A\rtimes_r G \ar[rr]^{\rho\rtimes_r id} && A/I\rtimes_r G \ar[r] & 0 }\]
(see \cite[Remark 7.14]{Williams}). The action of $G$ on $A$ is said to be \emph{exact} if $(\star)$ is exact for all $G$-invariant closed two-sided ideals in $A$ (\cite[Definition 1.2]{S}). Note that if $G$ is exact, then every continuous action of $G$ on $C(X)$ is exact. We also call the action $ G\curvearrowright X$ exact if it is induced by an exact action of $G$ on $C(X)$.

\begin{definition}(\cite{S})
A $C^\ast$-algebra $A$ is said to separate the ideals in $A\rtimes_r G$ if the (surjective) map $J\rightarrow J\cap A$, from the ideals in $A\rtimes_r G$ into the $G$-invariant ideals in $A$ is injective.
\end{definition}

It was shown in \cite{S} that if $C(X)$ separates ideals in $C(X)\rtimes_r G$ then the induced action of $G$ on $C(X)$ must be exact. In the converse direction, it was also shown in \cite{S} that if the action $\alpha: G\curvearrowright X$ is exact and essentially free then $C(X)$ separates ideals in $C(X)\rtimes_r G$.

For Cuntz comparison, we refer to \cite{A-P-T} as a reference. Let $A$ be a $C^\ast$-algebra. We write $M_\infty(A)=\bigcup_{n=1}^\infty M_n(A)$ (viewing $M_n(A)$ as an upper left-hand corner in $M_m(A)$ for $m>n$). Let $a,b$ be two positive elements in $M_n(A)_+$ and $M_m(A)_+$, respectively. Set $a\oplus b= \textrm{diag}(a,b)\in
M_{n+m}(A)_+$, and write $a\precsim_A b$ if there exists a sequence $(r_n)$ in $M_{m,n}(A)$ with $r_n^\ast
br_n\rightarrow a$.  If there is no confusion, we usually omit the subscript $A$ by writing $a\precsim b$ instead. We write $a\sim b$ if $a\precsim b$
and $b\precsim a$.   A non-zero positive element $a$ in $A$ is said to be\textit{ properly infinite} if
$a\oplus a\precsim a$. A $C^\ast$-algebra $A$ is said to be \textit{purely infinite} if there are no characters
on $A$ and if, for every pair of positive elements $a,b\in A$ such that $b$ belongs to the closed ideal in $A$ generated by $a$, one has $b\precsim a$. It was proved in \cite{K-R} that a $C^\ast$-algebra $A$ is purely infinite if and
only if every non-zero positive element $a$ in $A$ is properly infinite. To end this section we record the following proposition, which was proved by R{\o}rdam and Sierakowski in \cite{R-S}.

\begin{proposition}(\cite[Proposition 2.1]{R-S})
Let $A$ be a $C^*$-algebra and $G\curvearrowright A$ be a $C^\ast$-dynamical system with $G$ discrete. Suppose that $A$ separates the ideals in $A\rtimes_r G$. Then $A\rtimes_r G$ is purely infinite if and only if all non-zero positive elements in $A$ are properly infinite in $A\rtimes_r G$ and $E(a)\precsim a$ for all positive elements $a$ in $A\rtimes_r G$, where $E$ is the canonical conditional expectation from $A\rtimes_r G$ to $A$.
\end{proposition}

\section{Dynamical comparison and pure infiniteness}
In this section, under the hypothesis that there is no $G$-invariant regular Borel probability measure on $X$ we show that if the action $\alpha: G\curvearrowright X$ is topologically free and has dynamical comparison then the reduced crossed product $A=C(X)\rtimes_r G$ is simple and purely infinite. To do this, we follow the idea in \cite{L-S}. What we will actually show is the existence, for every nonzero element $x\in A$, of elements $y,z\in A$ such that $yxz=1_A$. In the simple case, this condition is well-known to be equivalent to the definition of pure infiniteness recalled in the previous section (see \cite[Proposition 4.1.1]{Rordam}).

\begin{definition}(\cite[Definition 1.1]{B-C})
An element $x$ in a $C^\ast$-algebra is called a \textit{scaling element} if $x^\ast x\neq xx^\ast$ and $(x^\ast x)(xx^\ast)=xx^\ast$.
\end{definition}

Note that if $x$ is a scaling element in a $C^\ast$-algebra $A$, then $v=x+(1-x^\ast x)^{1/2}$ is an isometry. To see this, it suffices to verify that $(1-x^\ast x)^{1/2}x=0$. Because $(x^\ast x)(xx^\ast)=xx^\ast$, one has $(1-x^\ast x)xx^\ast=(1-x^\ast x)|x^\ast|^{2}=0$, which implies that $(1-x^\ast x)^{1/2}|x^\ast|=0$ by functional calculus. Thus $(1-x^\ast x)^{1/2}x=(1-x^\ast x)^{1/2}|x^\ast|u=0$, where $x=u|x|=|x^\ast|u$ is the polar decomposition of $x$ in $A^{\ast\ast}$. Throughout the paper, for a function $f\in C(X)$, we denote by $\operatorname{supp}(f)$ the set $\operatorname{supp}(f)=\{x\in X: f(x)\neq 0\}$, which is an open subset of $X$. The following lemma strengthens Lemma 3 in \cite{L-S}.

\begin{lemma}
Suppose that $\alpha: G\curvearrowright X$ has dynamical comparison and there is no $G$-invariant regular probability Borel measure on $X$. Let $\phi\in C(X)$ be a non-zero positive function. Then there is an isometry $v\in C(X)\rtimes_r G$ such that $vv^\ast$ lies in the hereditary subalgebra $A(\phi)$ of $C(X)\rtimes_r G$ generated by $\phi$.
\end{lemma}
\begin{proof}
Choose $g\in C(X)$ with $0\leq g\leq 1$, $g=1$ on a neighborhood of $\phi^{-1}(\{0\})$, and
$\overline{\textrm{supp}(g)}\neq X$.  Let $U$ be open and nonempty with $\overline{U}\cap
\overline{\textrm{supp}(g)}=\emptyset$.  Let $V$ be open with $\overline{\textrm{supp}(g)}\subset V\subset
\overline{V}\subset \overline{U}^c$. Now, define $F=\overline{U\sqcup V}$ and we have $F\prec U$ since $\alpha$ has dynamical comparison. This means that there is an open cover $\mathcal{W}=\{W_1,\dots,W_n\}$ of $F$ and
$t_1,\dots,t_n\in G$ such that $\{t_iW_i: i=1,\dots, n\}$ contains pairwise disjoint subsets of $U$. Now, let $\{f_i:
i=1,2,\dots,n\}$ be a partition of unity subordinate to $\mathcal{W}$. We have
\begin{enumerate}
\item $0\leq f_i\leq 1$ for all $i=1,2,\dots,n$;
\item $\sum_{i=1}^n f_i(y)=1$ for all $y\in F$;
\item $\overline{\operatorname{supp}(f_i)}\subset W_i$ for all $i=1,2,\dots,n$.
\end{enumerate}
Define $x=\sum_{i=1}^n u_{t_i} f_i^{1/2}$. We claim that $x$ is a scaling element. At first, observe that $t_iW_i\cap t_jW_j=\emptyset$ whenever $i\neq j$. Therefore one has $f_i^{1/2}u_{t^{-1}_i}  u_{t_j} f_j^{1/2}=u_{t^{-1}_i} (u_{t_i} f_i^{1/2}u_{t^{-1}_i})  (u_{t_j} f_j^{1/2}u_{t^{-1}_j}) u_{t_j}=0$ if $i\neq j$. Then we have

\begin{align*}
x^\ast x&=(\sum_{i=1}^n f_i^{1/2}u_{t^{-1}_i})(\sum_{i=1}^n u_{t_i} f_i^{1/2})\\
&= \sum_{i=1}^n f_i+\sum_{1\leq i\neq j\leq n} f_i^{1/2}u_{t^{-1}_i}  u_{t_j} f_j^{1/2}\\
&= \sum_{i=1}^n f_i.\\
\end{align*}

and

\begin{align*}
xx^\ast&=(\sum_{i=1}^n u_{t_i} f_i^{1/2})(\sum_{i=1}^n f_i^{1/2}u_{t^{-1}_i})\\
&= \sum_{1\leq i,j\leq n} u_{t_i} f_i^{1/2} f_j^{1/2}u_{t^{-1}_j}\\
&= \sum_{1\leq i,j\leq n} (u_{t_i} f_i^{1/2} f_j^{1/2}u_{t^{-1}_i})  u_{t_i} u_{t^{-1}_j}.
\end{align*}

For all $i=1,2,\dots,n$ one has $\textrm{supp}(u_{t_i} f_i^{1/2} f_j^{1/2}u_{t^{-1}_i})\subset t_iW_i\subset U$. In addition $t_iW_i\subset U\subset F$ implies that  $\sum_{i=1}^n f_i(y)=1$ for every $y\in t_iW_i$. This implies that $(\sum_{i=1}^n f_i)(u_{t_i} f_i^{1/2} f_j^{1/2}u_{t^{-1}_i})= u_{t_i} f_i^{1/2} f_j^{1/2}u_{t^{-1}_i}$ for all $i=1,2,\dots,n$. Therefore, we have:
\begin{align*}
(x^\ast x)(xx^\ast)&=(\sum_{i=1}^n f_i)( \sum_{1\leq i,j\leq n} (u_{t_i} f_i^{1/2} f_j^{1/2}u_{t^{-1}_i})  u_{t_i} u_{t^{-1}_j})\\
&=( \sum_{1\leq i,j\leq n} (u_{t_i} f_i^{1/2} f_j^{1/2}u_{t^{-1}_i})  u_{t_i} u_{t^{-1}_j})\\
&= xx^\ast.
\end{align*}

If the set $\{t_i: i=1,2,\dots n\}$ contains at least two different group elements then $xx^\ast$ is not a function while $x^\ast x$ is. On the other hand, if there is a $t\in G$ such that  $t_i=t$ for every $i=1,2,\dots n$ then $xx^\ast= \sum_{1\leq i,j\leq n}u_{t} f_i^{1/2} f_j^{1/2}u_{t^{-1}}$, which is a function supported in $U$ while $x^\ast x$ is constant one on $F$. Therefore, in any case, one has $xx^\ast\neq x^\ast x$. These show that $x$ is a scaling element. Define an isometry $v=x+(1-x^\ast x)^{1/2}$ as mentioned above.

Observe that $1-x^\ast x=1-\sum_{i=1}^n f_i$ is constant zero on $F\supset \textrm{supp}(g)$. This implies that $g(1-x^\ast x)^{1/2}=0$.  In addition, for all $i=1,2,...,n$ one has
$gu_{t_i}f_i^{1/2}=g(u_{t_i}f_i^{1/2}u_{t^{-1}_i})u_{t_i}=0$ since supp$(u_{t_i}f^{1/2}_iu_{t^{-1}_i})\subset t_iW_i\subset U$. This implies that $gv=0$ and thus $gvv^\ast=0$.

Since $0\leq g\leq 1$ and $vv^\ast$ is a projection, one has $g+vv^\ast\leq
1$. Observe that supp$(1-g)\subset \textrm{supp}(\phi)$ so that  $1-g\precsim \phi$ in $C(X)$ in the sense of Cuntz comparison by Proposition 2.5 in \cite{A-P-T}. Hence $1-g\in A(\phi)$ since there is a sequence $\{r_n\}$
in $C(X)$ such that $\phi^{1/2}r_n^\ast r_n\phi^{1/2}=r_n^\ast\phi r_n\rightarrow 1-g$. Then because $vv^\ast\leq 1-g$, one has $vv^\ast\in A(\phi)$ by the definition of hereditary sub-algebras.

\end{proof}

Using the lemma above, the same proof of Theorem 5 in \cite{L-S} establishes Theorem 1.1. To be self-contained, we write the proof here. Let $S$ be a finite subset of $G$ and $O$ be a non-empty open subset of $X$. If the collection $\mathcal{T}=\{sO: s\in S\}$ is disjoint then we call $\mathcal{T}$ an \textit{open tower}, denoted by $(S,O)$.

\begin{proof}(Theorem 1.1)
Since the action $\alpha$ is minimal and topologically free, the reduced crossed product is simple. Therefore, it suffices to show that the reduced crossed product $A=C(X)\rtimes_r G$ is purely infinite. Let $x\in A$ with $x\neq 0$. We will find $y,z\in A$ with $yxz=1$. Observe that $E(x^\ast x)$ is a nonzero positive element in $C(X)$ since $E$ is the canonical faithful conditional expectation. Define $a=x^\ast x/\|E(x^\ast x)\|$. Then one has $a\geq 0$ and $\|E(a)\|=1$. Choose an element $b\in C_c(G, C(X))_+$ with $\|a-b\|<1/4$. Write $b=\sum_{t\in F}b_tu_t$ where $F$ is a finite subset of $G$ containing the identity element $e\in G$. Then $E(b)=b_e$ is a non-zero positive function and $\|E(b)\|>3/4$ because $\|E(b)-E(a)\|<1/4$.

Since the action $\alpha$ is topologically free, the open set $O=\{x\in X: tx\neq x\ \textrm{for all}\ t\in F^{-1}F\setminus \{e\}\}=\bigcap_{t\in F^{-1}F\setminus\{e\}}\{x\in X: tx\neq x\}$ is dense in $X$. Let $U_0$ be the non-empty open set of all $x\in X$ such that $E(b)(x)>3/4$. Choose an element $x_0\in U_0\cap O$ and a neighbourhood $U$ with $x_0\in U\subset U_0\cap O$ such that $(F,U)$ is an open tower. We can do this since the space $X$ is Hausdorff.

Choose $\phi\in C(X)$ with $0\leq \phi\leq 1$, supp$(\phi)\subset U$ and $\phi=1$ on a nonempty open set. Then we observe that $E(b)\geq (3/4)\phi$. Now let $\phi_1\in C(X)$ be another non-zero function, with $0\leq \phi_1\leq 1$ and supp$(\phi_1)\subset \phi^{-1}(\{1\})$. By Lemma 3.2 there is an isometry $v\in A$ with $vv^\ast\in A(\phi_1)$. We now claim that $v^\ast bv=v^\ast E(b)v$. To show this, first observe that $v^\ast bv=v^\ast (vv^\ast bvv^\ast)v$ since $v$ is an isometry. Then for every element of the form $\phi_1 a\phi_1$ in $A(\phi_1)$, one has
\begin{align*}
(\phi_1 a\phi_1) b(\phi_1 a\phi_1)=\sum_{t\in F} (\phi_1 a\phi_1) b_tu_t (\phi_1 a\phi_1)=(\phi_1 a\phi_1) E(b)(\phi_1 a\phi_1)
\end{align*}
since one can check that $\phi_1 b_tu_t\phi_1=b_t\phi_1 \cdot u_t \phi_1 u_{t^{-1}} u_t=0$ if $t\neq e$ by using the fact that supp$(\phi_1)$ and supp$(u_t \phi_1 u_{t^{-1}})$ are disjoint. Then since $vv^\ast\in A(\phi_1)$, one has $vv^\ast bvv^\ast= vv^\ast E(b)vv^\ast$. This proves the claim that $v^\ast bv=v^\ast E(b)v$. Using the same method and the fact that supp$(\phi_1)\subset \phi^{-1}(\{1\})$, one can also show that $v^\ast\phi v=v^*v=1$. Thus we have
\begin{align*}
v^\ast bv=v^\ast E(b)v\geq v^\ast(\frac{3}{4}\phi)v=\frac{3}{4}v^\ast v=\frac{3}{4}.
\end{align*}
Then $v^\ast av$ is invertible since $\|v^\ast av-v^\ast bv\|<1/4$. Let $y=\|E(x^\ast x)\|^{-1}(v^\ast av)^{-1}v^\ast x^\ast$ and $z=v$. Then we have $yxz=1_A$. Thus $A=C(X)\rtimes_r G$ is purely infinite.
\end{proof}

Based on Theorem 1.1, we also have the following corollary.

\begin{corollary}
Let $\alpha: G\curvearrowright X$ be an action on a compact metrizable space $X$ such that there is no $G$-invariant regular Borel probability measure on $X$. Suppose that the action $\alpha$ is topologically free, amenable and  has dynamical comparison. Then the reduced crossed product $C(X)\rtimes_r G$ is a Kirchberg algebra.
\end{corollary}

We close this section by remarking that reduced crossed products occurring in Example 2.1, 2.2 in \cite{L-S} and Example 2.1, 3.9, 4.3 in \cite{J-R} are covered by the corollary above since the actions are known to be topologically free, amenable, and $n$-filling for some integer $2\leq n\leq 6$ and to have no $G$-invariant regular Borel probability measures.

\section{Paradoxical comparison and pure infiniteness}
Beyond the issue of classification, whether a reduced crossed product is purely infinite is of its own interest.  In order to establish this pure infiniteness for a reduced crossed product one usually needs to formalize the phenomenon of paradoxicality in the framework of dynamical systems. Roughly speaking, the idea of paradoxicality dating back to the work of Hausdorff and playing an important role in the work of Banach-Tarski (see \cite{Wagon}), is that one object somehow contains two disjoint copies of itself.  The following notion introduced by R{\o}rdam and Sierakowski exactly follows this philosophy and is sufficient to show pure infiniteness of reduced crossed products if the space $X$ is zero-dimensional.  Motivated by their work,  we come up with another notion in this
section called \textit{paradoxical comparison}. This notion is weaker than dynamical comparison if the action is not minimal, but it still implies the pure infiniteness of the reduced crossed product if the action has an additional  property which we call the \textit{uniform tower property}. We recall a definition and a theorem of R{\o}rdam and Sierakowski first.

\begin{definition}\cite[Definition 4.2]{R-S}
Given a discrete group $\Gamma$ acting on a topological space $(Y, \tau_Y)$, a non-empty set $U$ is called $(\Gamma, \tau_Y)$-\textit{paradoxical} if there exist non-empty open sets $V_1, V_2,\dots, V_{n+m}$ and elements $t_1, t_2,\dots, t_{n+m}$ in $\Gamma$ such that
\begin{align*}
\bigcup_{i=1}^n V_i=\bigcup_{i=n+1}^{n+m} V_i=U
\end{align*}
and such that $(t_kV_k)_{k=1}^{n+m}$ are pairwise disjoint subsets of $U$.
\end{definition}

Using this notion, they obtained the following result.

\begin{theorem}\cite[Corollary 4.4]{R-S}
Let $\alpha: \Gamma\curvearrowright X$ with $\Gamma$ discrete and exact. Suppose that $\alpha$ is essentially free and $X$ has a basis of clopen $(G, \tau_X)$-\textit{paradoxical} sets. Then $C(X)\rtimes_r \Gamma$ is purely infinite.
\end{theorem}

For each nonempty open subset $O$ of $X$ we write $(O, O)\prec O$ if for every closed subset $F$ of $O$ there are disjoint nonempty open subsets $O_1$ and $O_2$ of $O$ such that $F\prec O_1$ and $F\prec O_2$. Similarly we write
$$(\underbrace{O, \dots , O}\limits_{n\  \textrm{many}})\prec O$$
if for every closed subset $F\subset O$ there are disjoint family of nonempty open subsets $O_1,\dots, O_n$ of $O$ such that $F\prec O_i$ for every $i=1,\dots, n$. Based on this notation, we arrive the following definition.

\begin{definition}
Let $\alpha: G\curvearrowright X$. We say that $\alpha$ has \textit{paradoxical comparison} if one has $(O, O)\prec O$ for every nonempty open subset $O$ of $X$.
\end{definition}

This definition also exactly follows the philosophy of paradoxicality since each open subset of $X$
contains two disjoint copies of itself in the sense of subequivalence and therefore it can be viewed as a dynamical
analogue of properly infiniteness of positive elements in $C^\ast$-setting.  In addition, we remark that an action
$\alpha: G\curvearrowright X$, where $X$ is zero dimensional, has paradoxical comparison if and only if every clopen subset of $X$ is $(G, \tau_X)$-paradoxical.
Indeed, first observe that a clopen subset of $X$ is $(G, \tau_X)$-paradoxical if and only if it satisfies the
condition of paradoxical comparison. Thus it suffices to show that if one has $(A,A)\prec A$ for every clopen
subset $A$ of $X$ then the action has paradoxical comparison. Let $F$ be a closed subset of an open set $O$. By compactness there is a clopen set $P$ such that $F\subset P\subset O$. Since $(P, P)\prec P$ one can find disjoint nonempty open subsets $O_{1}$ and $O_{2}$ of $P$ such that $F\prec O_{j}\subset O$ for $j=1,2$.  This verifies that the action $\alpha$ has paradoxical comparison. In light of Theorem 4.2, our paradoxical comparison then is also a candidate to show pure infiniteness of reduced crossed product in which the underlying space $X$ has a higher dimension.

We remark that if $\alpha: G\curvearrowright X$ has paradoxical comparison then $X$ has to be perfect because there is no two nonempty disjoint open subsets of an open set whose cardinality is one. In addition there is no $G$-invariant regular Borel probability measure on $X$. Indeed, suppose to the contrary that there is such a measure, say $\mu$.  For $X$ itself there are disjoint nonempty open subset $O_1$ and $O_2$ such that $X\prec O_i$ for $i=1,2$, which implies that $\mu(O_i)=1$ for $i=1,2$. Then one has $1=\mu(X)\geq \mu(O_1)+\mu(O_2)=2$, which is a contradiction. Furthermore, if the space $X$ is zero-dimensional then $\alpha: G\curvearrowright X$ has no $G$-invariant non-trivial Borel measure by applying the same argument to a clopen set $O$ with $0<\nu(O)<\infty$ to obtain a contradiction whenever there is such a measure $\nu$.

The following definition was suggested by David Kerr. We call this definition \textit{weak paradoxical comparison} in this paper. To justify this name, Proposition 4.6 below will show that paradoxical comparison implies weak paradoxical comparison. The reason we introduce this concept is that it helps in proving pure infiniteness of crossed products. See Proposition 4.11 below.

\begin{definition}
Let $\alpha: G\curvearrowright X$. We say $\alpha$ has \textit{weak paradoxical comparison} if for every closed subset $F$ and nonempty open subset $O$ of $X$ one has $F\prec O$ whenever $F\subset G\cdot O$.
\end{definition}

Before we prove Proposition 4.6, we need the following lemma which records elementary but useful properties of the relation of subequivalence.

\begin{lemma}
Let $\alpha: G\curvearrowright X$ be an action and $F$ a closed subset of $X$. Denote by $A, B, C, M, N$ nonempty open subsets of $X$. Then:
\begin{enumerate}[label=(\roman*)]
\item $F\prec A$ if and only if there is an open subset $M$ such that $F\prec M\subset \overline{M}\subset A$.

\item If $F\prec N\subset \overline{N}\prec B$ then $F\prec B$.

\item If $A\prec B$ and $B\prec C$ then $A\prec C$.
\end{enumerate}
\end{lemma}
\begin{proof}

For the claim (i) we begin with $F\prec A$. There are open sets $U_1,\dots, U_n$ and group elements $g_1,\dots, g_n\in G$ such that $F\subset \bigcup_{i=1}^n U_i$ and $\bigsqcup_{i=1}^n g_iU_i\subset A$. Then choose a partition of unity $\{f_1, \dots, f_n\}$ subordinate to the open cover $\{U_1, \dots, U_n\}$ of $F$ such that $\overline{\operatorname{supp}(f_i)}\subset U_i$ for all $i=1, \dots, n$. Define $W_i=\operatorname{supp}(f_i)$ for each $i$. Then $\{W_i: i=1, \dots, n\}$ also forms an open cover of $F$ and $\bigsqcup_{i=1}^n g_i\overline{W_i}\subset A$. Define $M=\bigsqcup_{i=1}^n g_iW_i$ and thus $\overline{M}=\bigsqcup_{i=1}^n g_i\overline{W_i}$, which is a closed subset of $A$. The converse is trivial.

For the claim (ii) suppose that $F\prec N\subset \overline{N}\prec B$ holds. Then there are open sets $O_1,\dots, O_n$ and group
elements $g_1,\dots, g_n\in G$ such that $F\subset \bigcup_{i=1}^n O_i$ and $\bigsqcup_{i=1}^n g_iO_i\subset
N$. In addition, for $\overline{N}\prec B$ there are open sets $U_1,\dots, U_m$ and group elements $h_1,\dots, h_m\in G$ such that
$\overline{N}\subset \bigcup_{j=1}^m U_j$ and $\bigsqcup_{j=1}^m h_jU_j\subset B$. Observe that $\bigsqcup_{i=1}^n g_iO_i\subset N\subset \bigcup_{j=1}^m
U_j$. Then $\{O_i\cap g^{-1}_iU_j: i=1, \dots, n, j=1,\dots, m \}$ form a cover of $F$ and $\{h_jg_i\cdot
(O_i\cap g_i^{-1}U_j)= h_j(g_iO_i\cap U_j): i=1, \dots, n, j=1,\dots, m\}$ is disjoint in $B$. This shows that $F\prec B$.

The claim (iii) follows from the two claims before. Since one has $A\prec B$, for every closed subset $F$ of $A$ there is an open subset $M$ such that $F\prec M\subset \overline{M}\subset B\prec C$. Then claim (ii) implies that $F\prec C$. Then $A\prec C$ since $F$ is arbitrary.
\end{proof}

\begin{proposition}
Let $\alpha: G\curvearrowright X$ be an action such that there is no $G$-invariant regular Borel probability measure on $X$.  Consider the following properties:
\begin{enumerate}[label=(\roman*)]
\item $\alpha$ has dynamical comparison;

\item $\alpha$ has paradoxical comparison;

\item $\alpha$ has weak paradoxical comparison;
\end{enumerate}
Then (i)$\Rightarrow$(ii)$\Rightarrow$(iii). In addition, if $\alpha$ is minimal then these three conditions are equivalent.
\end{proposition}
\begin{proof}
(i)$\Rightarrow$(ii) Let $F$ be a closed subset and $O$ an open subset such that $F\subset O$. Since the space $X$ is Hausdorff and perfect, there are nonempty disjoint open subset $O_1, O_2$ of $O$. Observe that $O\prec O_i$ for $i=1,2$ since the action has dynamical comparison. Then $F\prec O_i$ for $i=1,2$.

(ii)$\Rightarrow$(iii). Suppose that $\alpha: G\curvearrowright X$ has paradoxical comparison. Now given a closed subset $K$ and an open subset $O$ of $X$ such that $K\subset G\cdot O$.  Then there is a finite subset $E$ of $G$ such that $K\subset \bigcup_{h\in E}h\cdot O$. Let $n=|E|$. We first claim
$$(\underbrace{O, \dots , O}\limits_{n\  \textrm{many}})\prec O.$$
Indeed, let $F$ be a closed subset of $O$ and $k$ an integer such that $2^k\geq n$. By induction on $k$ we construct two collections of open subsets of $O$, say $\{M_{i_1i_2\dots i_m}: i_1, \dots, i_m=1,2\ \textrm{and}\ 1\leq m\leq k\}$ and $\{O_{i_1i_2\dots i_m}: i_1, \dots, i_m=1,2\ \textrm{and}\ 1\leq m\leq k\}$ such that
\begin{enumerate}
\item $F\prec M_i$ for $i=1,2$;

\item for every $1\leq m\leq k-1$ and $i_{m+1}=1,2$, one has $\overline{M_{i_1i_2\dots i_m}}\prec M_{i_1i_2\dots i_mi_{m+1}}$;
\item $\overline{M_{i_1i_2\dots i_m}}\subset O_{i_1i_2\dots i_m}$ for any integer $m\in [1, k]$  and $i_1, i_2,\dots, i_m=1,2$;
\item for any integer $m\in [1, k]$ the collection $\{O_{i_1i_2\dots i_m}: i_1, \dots, i_m=1,2\}$ is disjoint.
\end{enumerate}

To do this, since $\alpha: G\curvearrowright X$ has paradoxical comparison, $(O, O)\prec O$ implies that for $F$ there are nonempty disjoint open subsets $O_1$ and $O_2$ of $O$ such that $F\prec O_i$ for $i=1,2$. Then for each $i$ there is an open subset $M_i$ such that $F\prec M_i\subset \overline{M_i}\subset O_i$  by Lemma 4.5(i). Then for each $i=1,2$, because  $(O_i, O_i)\prec O_i$, for
$\overline{M_i}$ one can find disjoint nonempty open subsets $O_{i1}$ and $O_{i2}$ of $O_i$ such that
$\overline{M_i}\prec O_{ij}$ for $j=1, 2$. Then Lemma 4.5(i) again implies that there are open subsets $M_{ij}$
such that $\overline{M_i}\prec M_{ij}\subset \overline{M_{ij}}\subset O_{ij}$ for $i,j=1,2$. Then suppose that we have obtained $\{M_{i_1i_2\dots i_m}: i_1, \dots, i_m=1,2\ \textrm{and}\ 1\leq m\leq l\}$ and $\{O_{i_1i_2\dots
i_m}: i_1, \dots, i_m=1,2\ \textrm{and}\ 1\leq m\leq l\}$ for $l<k$ so that they satisfies the conditions above.
Then since the action has paradoxical comparison, for each $\overline{M_{i_1i_2\dots i_l}}\subset O_{i_1i_2\dots
i_l}$ there are disjoint nonempty open subsets $O_{i_1i_2\dots i_li_{l+1}}$ of $O_{i_1i_2\dots i_l}$ such that
$\overline{M_{i_1i_2\dots i_l}}\prec O_{i_1i_2\dots i_li_{l+1}}$ where $i_{l+1}=1,2$. Then Lemma 4.5(i) entails
that there are open subsets $M_{i_1i_2\dots i_li_{l+1}}$ such that $\overline{M_{i_1i_2\dots i_l}}\prec
M_{i_1i_2\dots i_li_{l+1}}\subset \overline{M_{i_1i_2\dots i_li_{l+1}}}\subset O_{i_1i_2\dots i_li_{l+1}}$.  Observe that $\{O_{i_1i_2\dots i_{l+1}}: i_1, \dots, i_{l+1}=1,2\}$ is indeed disjoint. This finishes our construction, from which for $i_1,\dots, i_k=1,2$, we have
\[F\prec M_{i_1}\subset \overline{M_{i_1}}\prec M_{i_1i_2}\subset \overline{ M_{i_1i_2}}\prec\dots\prec M_{i_1i_2\dots i_k}.\]
Now we write $\{U_1,\dots, U_{2^k}\}$ for the disjoint collection $\{M_{i_1i_2\dots i_k}: i_1, \dots, i_k=1,2\}$. Then (ii) in Lemma 4.5 implies that $F\prec U_i$ for all $1\leq i\leq
2^k$. This shows the claim since $2^k\geq n$.

Now write $E=\{h_1,\dots,h_n\}$ and $K\subset \bigcup_{i=1}^nh_iO$. Then by the partition of unity argument exactly used in the proof of Lemma 4.5(i) there are open subsets $W_i\subset \overline{W_i}\subset h_iO$ for $i=1,\dots, n$ such that $K\subset \bigcup_{i=1}^n W_i$. Define $V_i=h_i^{-1}W_i$ and thus
$\overline{V_i}=h_i^{-1}\overline{W_i}$. This implies that $K\subset \bigcup_{i=1}^n h_i\overline{V_i}$ where
$\overline{V_i}\subset O$ for each $i=1,\dots, n$. Define a closed subset $F'=\bigcup_{i=1}^n \overline{V_i} \subset O$. Now consider
$$(\underbrace{O, \dots , O}\limits_{n\  \textrm{many}})\prec O.$$
Then there is a collection of disjoint open subsets $\{O_i: i=1,\dots, n\}$ such that $F'\prec O_i$ for each $i=1,\dots, n$. Then for the collection $\{\overline{V_i}: i=1,\dots, n\}$ there is a collection of open subsets $\{U_j^{(i)}: j=1,\dots, k_i, i=1,\dots, n \}$ and group elements
$\{g_j^{(i)}\in G: j=1,\dots, k_i, i=1,\dots, n \}$ such that $\overline{V_i}\subset F'\subset
\bigcup_{j=1}^{k_i}U_j^{(i)}$ and $\bigsqcup_{j=1}^{k_i}g_j^{(i)}U_j^{(i)}\subset O_i$ for each $i=1,\dots, n$. This implies that the collection of open subsets $\{g^{(i)}_jU^{(i)}_j: j=1,\dots, k_i,
i=1,\dots, n\}$ is disjoint in $O$. Therefore,  $\{h_iU_j^{(i)}: j=1,\dots, k_i, i=1,\dots, n \}$ form an open
cover of $K$ and $\{g^{(i)}_jh^{-1}_i\cdot(h_iU^{(i)}_j): j=1,\dots, k_i, i=1,\dots, n\}$ is a disjoint
collection of open subsets of $O$. This verifies $K\prec O$.

(iii)$\Rightarrow$(i)(if the action is minimal). It suffices to show that for every nonempty open subsets $A, B$ of $X$ one has $A\prec B$. Indeed for every closed subset $F\subset A$ one always has $F\subset G\cdot B=X$ since the action $\alpha: G\curvearrowright X$ is minimal. Then $F\prec B$ because the action has weak paradoxical comparison. Therefore one has $A\prec B$ since $F$ is arbitrary.
\end{proof}

On the other hand, to make the proposition above more sense we need to show that, unlike dynamical comparison,  paradoxical comparison does not necessarily imply that the action is minimal. Otherwise, paradoxical comparison is equivalent to dynamical comparison in general and it suffices to apply Theorem 1.1 to establish the pure infiniteness of a reduced crossed product from paradoxical comparison.  We will construct an explicit example (Example 5.4 below) in the next section in which the action has paradoxical comparison but is not minimal. In the rest of this section we will show reduced crossed products is purely infinite if the action has the paradoxical comparison and the following property.

\begin{definition}
We say an action $\alpha: G\curvearrowright X$ has the \textit{uniform tower property} if for all open subsets  $O, U$ of $X$ such that $\overline{O}\subset U$ and all finite subsets $T$ of $G$ there is a nonempty closed set $F$ and an open set $W$ with $F\subset W\subset U$ such that
\begin{enumerate}[label=(\roman*)]
\item $(T, W)$ is an open tower;

\item $O\cap Y\neq \emptyset$ implies that $F\cap Y\neq \emptyset$ for all $G$-invariant closed subsets $Y$ of $X$.
\end{enumerate}

\end{definition}

Note that if an action is minimal and topologically free then it has the uniform tower property trivially.  The following lemma is a generalization of Lemma 3.4 in \cite{H}.

\begin{lemma}
Let $\alpha: G\curvearrowright X$. For non-zero positive functions $f,g\in C(X)_+$ and $\epsilon>0$, if $\overline{\operatorname{supp}((f-\epsilon)_+)}\prec \operatorname{supp}(g)$ then $(f-\epsilon)_+)\precsim g$ in $C(X)\rtimes_r G$ and if  $\operatorname{supp}(f)\prec\operatorname{supp}(g)$ then $f\precsim g$ in $C(X)\rtimes_r G.$

\end{lemma}
\begin{proof}
Suppose we have $\overline{\operatorname{supp}((f-\epsilon)_+)}\prec \operatorname{supp}(g)$. Then we have a family $\mathcal{U}=\{U_1, U_2, \dots, U_n\}$  of open sets forming a cover of
$\overline{\textrm{supp}((f-\epsilon)_+)}$ and elements $\gamma_1,\gamma_2,\dots, \gamma_n\in G$ so
that $\{\gamma_iU_i: i=1,2,\dots,n\}$ is a disjoint family of open subsets of supp$(g)$. Let $\{f_i:
i=1,2,\dots, n\}$ be a partition of unity subordinate to $\mathcal{U}$ so that
\begin{enumerate}
	\item $0\leq f_i\leq 1$ for all $i=1,2,...,n$;
	\item $\sum_{i=1}^n f_i(x)=1$ for all $x\in \overline{\textrm{supp}((f-\epsilon)_+)}$;
	\item $\overline{\operatorname{supp}(f_i)}\subset U_i$ for all $i=1,2,...,n$.
\end{enumerate}
Then we have supp$((f-\epsilon)_+)\subset \textrm{supp}(\sum_{i=1}^n f_i)$ and this implies that
$(f-\epsilon)_+\precsim _{C(X)} \sum_{i=1}^n f_i$. Define $u=\bigoplus_{i=1}^n u_{\gamma_i}$. We have
\begin{align*}
\sum_{i=1}^n f_i\precsim \bigoplus_{i=1}^n f_i\sim u(\bigoplus_{i=1}^n f_i)u^\ast=\bigoplus_{i=1}^n
\alpha_{\gamma_i}(f_i).
\end{align*}
Then since supp$(\alpha_{\gamma_i}(f_i))\subset \gamma_iU_i$ for every $i=1,2,\dots n$
and supp$(\sum_{i=1}^n\alpha_{\gamma_i}(f_i))\subset \bigsqcup_{i=1}^n \gamma_iU_i\subset \textrm{supp}(g)$, we
have
\begin{align*}
\bigoplus_{i=1}^n \alpha_{\gamma_i}(f_i)\sim \sum_{i=1}^n \alpha_{\gamma_i}(f_i)\precsim g.
\end{align*}
Therefore, we have $(f-\epsilon)_+\precsim g$ in $C(X)\rtimes_r G$.	
	
Now suppose  $\operatorname{supp}(f)\prec\operatorname{supp}(g)$ holds.  In order to show $f\precsim g$ in $C(X)\rtimes_r G$ it suffices to show that $(f-\epsilon)_+\precsim g$ for all $\epsilon>0$ by Proposition 2.17 in \cite{A-P-T}.  Observe that for $f\in C(X)$ and $F\in C_0((0, \|f\|_{\infty}])_+$  one has $F(f)(x)=F(f(x))$ by functional calculus. Therefore, $(f-\epsilon)_+(x)=f(x)-\epsilon$ if $f(x)\geq \epsilon$ while $(f-\epsilon)_+(x)=0$ if $f(x)< \epsilon$.	

For every $\epsilon>0$, define $C_\epsilon=\{x\in X: f(x)\geq \epsilon\}$. Then
$\overline{\operatorname{supp}((f-\epsilon)_+)}\subset C_\epsilon\subset \operatorname{supp}(f)$, which entails that
$\overline{\operatorname{supp}((f-\epsilon)_+)}\prec \operatorname{supp}(g)$ since $\operatorname{supp}(f)\prec\operatorname{supp}(g)$.
Then the result above shows that $(f-\epsilon)_+\precsim g$. This implies that $f\precsim g$ in $C(X)\rtimes_r G$ because the $\epsilon$ is arbitrary.
\end{proof}

\begin{proposition}
Suppose that the action $\alpha: G\curvearrowright X$  has paradoxical comparison. Then $f\oplus f\precsim f$ in $C(X)\rtimes_r G$ for every non-zero function $f\in C(X)_+$.
\end{proposition}
\begin{proof}
Let $f$ be a non-zero element in $C(X)_+$ and $\epsilon>0$. Denote by $O=\operatorname{supp}(f)$ and $F=\overline{\operatorname{supp}(f-\epsilon)_+}$. Then there are nonempty disjoint open subsets $O_1, O_2$ of $O$ such that $F\prec O_1$ and $F\prec O_2$. Using $X$ is perfectly normal choose two positive functions $h_1, h_2\in C(X)$ such that $\operatorname{supp}(h_i)= O_i$ for $i=1,2$.
Then one has $(f-\epsilon)_+\precsim h_i$ for $i=1,2$ by Lemma 4.8. This implies that $(f-\epsilon)_+\oplus (f-\epsilon)_+\precsim h_1\oplus
  h_2\sim h_1+h_2\precsim f$ in $C(X)\rtimes_r G$. Thus, by Proposition 3.3 in \cite{K-R} one has $f\oplus f\precsim f$ since the $\epsilon$ is arbitrary.
\end{proof}

\begin{lemma}
Suppose that an action $\alpha: G\curvearrowright X$ has weak paradoxical comparison. Let $F$ be a closed subset and $O$ an open subset of $X$ . Suppose that  $F\cap Y\neq \emptyset$ implies $O\cap Y\neq \emptyset$ for all closed $G$-invariant subspaces $Y$. Then $F\prec O$.
\end{lemma}
\begin{proof}
Since the action has weak paradoxical comparison, it suffices to verify $F\subset G\cdot O$. Indeed, let $x\in F$ and define $Y=\overline{G\cdot x}$. Now we have $F\cap Y\neq \emptyset$ and thus $O\cap Y\neq \emptyset$ holds by the assumption. This implies that there is a $g\in G$ such that $gx\in O$ which implies that $x\in G\cdot O$. Since $x$ is an arbitrary element of $F$, one has $F\subset G\cdot O$.
\end{proof}

The proof of the following proposition contains ideas from Lemma 7.8 and 7.9 in \cite{NCP}.  Recall that for $f\in C(X)_+$ and $\epsilon>0$ we have $(f-\epsilon)_+(x)=\max\{f(x)-\epsilon, 0\}.$

\begin{proposition}
Suppose that an action $\alpha: G\curvearrowright X$ has weak paradoxical comparison as well as the uniform tower property. Then $E(a)\precsim a$ in $C(X)\rtimes_r G$ for every positive $a\in C(X)\rtimes_r G$.
\end{proposition}
\begin{proof}
It suffices to show the case that $a$ is a non-zero positive element in $C(X)\rtimes_r G$ with $\|a\|=1$. Observe that $E(a)\neq 0$ since $E$ is faithful. Define
$O=\operatorname{supp}(E(a))$. Fix an $\epsilon\in(0, \|E(a)\|)$ and define $U=\operatorname{supp}(E(a)-\epsilon)_+=\{x\in X : E(a)(x)>\epsilon\}$. Then choose a $\delta\in (0, \epsilon/4)$ and
a $c\in C_c(G, C(X))$ with $\|c\|\leq 2$ and $\|c-a^{\frac{1}{2}}\|<\frac{\delta}{8}$. This implies that
\begin{align*}
\|c^\ast c-a\|
&\leq\|c^\ast-a^{\frac{1}{2}}\|\|c\|+\|a^{\frac{1}{2}}\|\|c-a^{\frac{1}{2}}\|<\frac{3\delta}{8}<\frac{\delta}{2};
\end{align*}
and
\begin{align*}
\|cc^\ast -a\|&\leq \|c-a^{\frac{1}{2}}\|\|c^*\|+\|a^{\frac{1}{2}}\|\|c^*-a^{\frac{1}{2}}\|<\frac{3\delta}{8}<\frac{\delta}{2}.
\end{align*}

We write $b=c^*c=\sum_{t\in T}b_tu_t$, where $T$ is a finite subset of $G$. Since $b$ is positive non-zero element in $C(X)\rtimes_r G$ and the canonical conditional expectation $E$ is faithful, one has $E(b)=b_e\neq 0$ and $e\in T$.  We also observe that $\|E(b)-E(a)\|<\delta/2$, which implies
that $\overline{U}\subset\{x\in X: E(a)(x)>\frac{\epsilon}{2}\}\subset \{x\in X: E(b)(x)>\frac{\epsilon}{2}-\frac{\delta}{2}\}$. We
write $M$ for the open subset $\{x\in X: E(b)(x)>\frac{\epsilon}{2}-\frac{\delta}{2}\}$ for simplicity. One observes that
$M\subset \{x\in X: E(b)(x)\geq\frac{\epsilon}{2}-\frac{\delta}{2}\}\subset \{x\in X: E(b)(x)>\frac{\epsilon}{2}-\delta\}$.

Now apply the uniform tower property to  $\overline{M} \subset \{x\in X: E(b)>\frac{\epsilon}{2}-\delta\}$ so that one obtains a nonempty closed set $F$ and an open set $W$ with $F\subset W\subset  \{x\in X: E(b)>\frac{\epsilon}{2}-\delta\}$ such that $(T, W)$ is a tower and
$$ M\cap Y\neq \emptyset\Longrightarrow F\cap Y\neq \emptyset$$
for all closed $G$-invariant subsets $Y$ of $X$.

Then choose a continuous function $f\in C(X)$ satisfying
\[0\leq f\leq 1,\ \ \operatorname{supp}(f)\subset W,\ \ \textrm{and}\ \ f|_F\equiv 1.\]
Then one has
\[fbf=fE(b)f+\sum_{t\in T\setminus\{e\}}fb_tu_tf=fE(b)f+\sum_{t\in T\setminus\{e\}}fb_t\alpha_t(f)u_t.\]
Since $\operatorname{supp}(\alpha_t(f))\subset tW$ and $\{tW: t\in T\}$ is an open tower, $fb_t\alpha_t(f)=b_tf\alpha_t(f)=0$ whenever $t\neq e$. This entails that $$fbf=fE(b)f\in C(X)_+.$$
In addition, since $F\subset \{x\in X: E(b)>\frac{\epsilon}{2}-\delta\}$, for every $x\in F$ one has $(fE(b)f)(x)=E(b)(x)>\frac{\epsilon}{2}-\delta>\delta$ by our choice of $\delta$. This implies that $F\subset \operatorname{supp}((fE(b)f-\delta)_+)$. Thus $F\cap Y\neq \emptyset$ implies that $\operatorname{supp}((fE(b)f-\delta)_+)\cap Y\neq \emptyset$ for all closed $G$-invariant subspaces $Y$ of $X$. Therefore, by the argument above we have
$$ \overline{U}\cap Y\neq \emptyset\Longrightarrow M\cap Y\neq \emptyset\Longrightarrow F\cap Y\neq \emptyset\Longrightarrow \operatorname{supp}((fE(b)f-\delta)_+)\cap Y\neq \emptyset$$
for all closed $G$-invariant subspaces $Y$ of $X$.
Then Lemma 4.10 implies that
$$\overline{\operatorname{supp}(E(a)-\epsilon)_+}=\overline{U}\prec \operatorname{supp}((fE(b)f-\delta)_+).$$
Now Lemma 4.8 entails that
\[(E(a)-\epsilon)_+\precsim (fE(b)f-\delta)_+=(fbf-\delta)_+.\]

On the other hand, Lemmas 1.4 and 1.7 in \cite{NCP} imply that
\[(fbf-\delta)_+=(fc^*cf-\delta)_+\sim(cf^2c^*-\delta)_+\precsim (cc^*-\delta)_+\precsim a.\]

Therefore, we have $(E(a)-\epsilon)_+\precsim a$ in $C(X)\rtimes_r G$. Since $\epsilon$ is arbitrary one has $E(a)\precsim a$ as desired.
\end{proof}

Now, we are able to prove Theorem 1.3.

\begin{proof}(Theorem 1.3)
Suppose that the action $\alpha: G\curvearrowright X$ is  exact and essentially free. In addition, suppose that $\alpha$ has paradoxical comparison as well as the uniform tower property. It was shown in \cite{S} that if the group action $\alpha: G\curvearrowright X$ is exact and  essentially free then $C(X)$ separates ideals in $C(X)\rtimes_r G$. In addition, by Proposition 4.9 and 4.11, we have verified that all non-zero positive elements in $C(X)$ are properly infinite in $C(X)\rtimes_r G$ and $E(a)\precsim a$ for all positive elements $a$ in $C(X)\rtimes_r G$. Then Proposition 2.6 implies that the reduced crossed product $C(X)\rtimes_r G$ is purely infinite.
\end{proof}

\section{Applications and Examples}
In this section we will provide some applications of Theorem 1.3 by proving Corollary 1.4 and Corollary 1.5. We start with Corollary 1.4.

\begin{proof}(Corollary 1.4)
Recall the setting that the action $\alpha: G\curvearrowright X$ is exact and essentially free. In addition,  we assume that it has paradoxical
comparison. Then to show pure infiniteness by Theorem 1.3 it suffices to show that the action $\alpha: G\curvearrowright X$ has the uniform tower property. To this end, we begin with open sets $O, U$ such that $\overline{O}\subset U$ and a finite subset
$T$ of $G$. Since there are only finitely many $G$-invariant closed subsets of $X$, the set $\mathcal{I}=\{Y\subset X: O\cap Y\neq \emptyset,\ Y\ \textrm{closed\ and}\ G\cdot Y=Y\}$ has minimal elements with respect to the partial order
``$\subset$'', where a minimal element $Y\in \mathcal{I}$ means that there is no $G$-invariant subset $Z\in
\mathcal{I}$ such that $Z\subsetneq Y$. Denote by $\{Y_1,\dots, Y_m\}$ the set of all minimal elements in $\mathcal{I}$. Then
we claim $O\cap (Y_i\setminus \bigcup_{j\neq i} Y_j)\neq \emptyset$ for each $i=1,\dots, m$. Suppose not, let
$O\cap (Y_i\setminus \bigcup_{j\neq i} Y_j)= (O\cap Y_i)\setminus\bigcup_{j\neq i} Y_j=\emptyset$ for some $i$.
This implies that $\emptyset \neq O\cap Y_i\subset \bigcup_{j\neq i} Y_j$ and thus $O\cap Y_i\cap Y_j\neq
\emptyset$ for some $j\neq i$. However, this implies that $Y_i\cap Y_j\in \mathcal{I}$, which is a contradiction
to the minimality of $Y_i$ and $Y_j$ in $\mathcal{I}$. This shows the claim.

Define $D_T=\{x\in X: tx\neq x\ \textrm{for all}\ t\in T^{-1}T\setminus\{e\}\}=\bigcap_{t\in T^{-1}T\setminus\{e\}}\{x\in X: tx\neq x\}$.
Since the action $\alpha: G\curvearrowright X$ is essentially free, $D_T\cap Y$ is open dense in $Y$ with respect to the relative topology for all
$G$-invariant proper closed subset $Y$ of $X$. From the claim above we see $O\cap (Y_i\setminus \bigcup_{j\neq i} Y_j)$ is a non-empty relatively open subset of $Y_i$ and thus $M_{i, T}=D_T\cap O\cap (Y_i\setminus
\bigcup_{j\neq i} Y_j)\neq \emptyset$. Now choose $x_i\in M_{i, T}$ for each $i=1,\dots, m$. Since each
$Y_i\setminus \bigcup_{j\neq i} Y_j$ is a $G$-invariant subset, the points in $\{tx_i: i=1,\dots, m, t\in T\}$
are pairwise different. Then since the space is Hausdorff, there is a disjoint collection of open subsets of $X$, say $\{O_{tx_i}\subset X: i=1,2,\dots, m, t\in T\}$ such that
$tx_i\in O_{tx_i}$ for $t\in T$ and $i=1,\dots, m$. Now define $W_i=\bigcap_{t\in T}t^{-1}O_{tx_i}\cap O$ for
$i=1,\dots, m$. Then $(T, W_i)$ form an open tower and $TW_i\cap TW_j=\emptyset$ for $1\leq i\neq j\leq m$. In addition we may choose a closed subset $F_i$ of $X$ such that $x_i\in F_i\subset W_i$ for
$i=1,\dots, m$ by normality of the space $X$. Now define $W=\bigsqcup_{i=1}^m W_i\subset O\subset U$ and
$F=\bigsqcup_{i=1}^m F_i$. Then $(T, W)$ form an open tower by our construction. In addition, let $Y\in
\mathcal{I}$. Then there is a minimal element $Y_i\in \mathcal{I}$ such that $Y_i\subset Y$ where $1\leq i\leq
m$. Then $F\cap Y_i\neq \emptyset$ by our construction and thus $F\cap Y\neq \emptyset$. This shows that the
action $\alpha: G\curvearrowright X$ has the uniform tower property.

On the other hand, since the action $\alpha: G\curvearrowright X$ is exact and essentially free, $C(X)$ separates ideals in the crossed product $C(X)\rtimes_r G$. Therefore the number of $G$-invariant closed subsets is equal to the number of ideals in $C(X)\rtimes_r G$ and thus the crossed product has finitely many ideals.
\end{proof}

To show Corollary 1.5 we need following propositions.  Denote by $\pi_X, \pi_Y$ projection maps from $X\times Y$ to $X$ and $Y$ respectively.

\begin{proposition}
Let $\beta: G\curvearrowright X$ be a minimal topologically free action that has no $G$-invariant regular Borel probability measure. Suppose that $\beta$ has dynamical comparison. Let $Y$ be another compact Hausdorff space. Let $\alpha: G\curvearrowright X\times Y$ be an action defined by $\alpha_g((x, y))=(\beta_g(x),y)$.

\begin{enumerate}[label=(\roman*)]
\item  If $M\subset X\times Y$ is a $G$-invariant closed subset of the action $\alpha$ then $M=X\times \pi_Y(M)$.

\item  The action $\alpha$ is essentially free.
\end{enumerate}
\end{proposition}
\begin{proof}
For (i) it suffices to show $X\times \pi_Y(M)\subset M$ since the converse direction is trivial. Fix a $y\in \pi_Y(M)$. For every $x\in X$ and every neighbourhood $O$ of $x$, there is a $g\in G$ such that $\beta_g(x)\in O$. This implies that $\alpha_g((x,y))=(\beta_g(x),y)\in O\times \{y\}$ and thus the restriction of $\alpha$ on $X\times \{y\}$ is minimal with respect to the relative topology. Then since $M\cap (X\times \{y\})$ is a closed $G$-invariant subset of $X\times \{y\}$, one has $M\cap (X\times \{y\})=X\times \{y\}$ and thus $X\times \{y\}\subset M$. Therefore one has $X\times \pi_Y(M)\subset M$.

For (ii) it suffices to show that the action $\alpha$ is topologically free when restricted to any $G$-invariant closed subset $X\times P$ for some closed $P\subset Y$ by (i). Indeed, for each $g\in G$, one has:
\[\{(x,y)\in X\times P: \alpha_g(x,y)=(x,y)\}=\{x\in X: \beta_g(x)=x\}\times P\]
whose interior in $X\times P$ is empty since the interior of $\{x\in X: \beta_g(x)=x\}$ is empty in $X$. This shows that action $\alpha$ is topologically free on $X\times P$ and thus $\alpha$ is essentially free.
\end{proof}

\begin{proposition}
Suppose that $\alpha: G\curvearrowright X\times Y$ is the action in Proposition 5.1. Then $\alpha$ has paradoxical comparison.
\end{proposition}
\begin{proof}
Let $O$ be an open subset and $F$ be a closed subset of $X\times Y$ such that $F\subset O$. For all $(x,y)\in F$ there is an open
neighbourhood $M_x\times N_y$ of $(x,y)$ such that $(x,y)\in M_x\times N_y\subset O$. All of these neighbourhoods
form an open cover of $F$ so that we can choose a finite subcover, say $F\subset\bigcup_{i=1}^m M_i\times N_i$.
Then by the argument of partition of unity, there is a collection of closed subsets $\{F_i: i=1,\dots, m\}$ such that $F_i\subset M_i\times N_i$ and $F\subset \bigcup_{i=1}^m
F_i$. Then since the space $X$ is perfect we choose a collection of different points $\{x_{ij}\in M_i: i=1,\dots, m, j=1,2\}$. Then since $X$ is Hausdorff there is a collection of disjoint open sets
$\{O_{ij}: i=1,\dots,m, j=1,2\}$ such that $x_{ij}\in O_{ij}$ for each $i, j$. For each $i,j$ we may assume $O_{ij}\subset M_i$ by redefining $O_{ij}\defeq O_{ij}\cap M_i$. Now for $j=1,2$ we define $O_j=\bigsqcup_{i=1}^m O_{ij}\times N_i\subset O$. Then it suffices to verify $F\prec O_j$ for $j=1,2$.

Now fix $j\in \{1,2\}$. For each $i=1,\dots, m$, since $F_i\subset M_i\times N_i$ one has $\pi_X(F_i)$ is a compact
subset of $M_i$. Since $\beta: G\curvearrowright X$ has dynamical comparison, one has $M_i\prec O_{ij}$ for each
$i=1,\dots, m$, which means that there is a collection of open subsets of $X$, $\{P^{(i)}_1, \dots, P^{(i)}_{n_i}\}$ and a collection
of group elements $\{g_1^{(i)},\dots, g_{n_i}^{(i)}\}$ such that $\pi_X(F_i)\subset \bigcup_{k=1}^{n_i}P_k^{(i)}$
and $\bigsqcup_{k=1}^{n_i}\beta_{g_k^{(i)}}(P_k^{(i)})\subset O_{ij}$ for $i=1,\dots, m$. Then
one has
\[F_i\subset \pi_X(F_i)\times \pi_Y(F_i)\subset (\bigcup_{k=1}^{n_i}P_k^{(i)})\times N_i=\bigcup_{k=1}^{n_i}(P_k^{(i)}\times N_i);\]
while
\[\bigsqcup_{k=1}^{n_i}\alpha_{g_k^{(i)}}(P_k^{(i)}\times N_i)=\bigsqcup_{k=1}^{n_i}(\beta_{g_k^{(i)}}(P_k^{(i)})\times N_i)\subset O_{ij}\times N_i.\]
Therefore, one has:
\[F\subset \bigcup_{i=1}^mF_i\subset\bigcup_{i=1}^m\bigcup_{k=1}^{n_i}(P_k^{(i)}\times N_i); \]
while
\[\bigsqcup_{i=1}^m\bigsqcup_{k=1}^{n_i}\alpha_{g_k^{(i)}}(P_k^{(i)}\times N_i)\subset \bigsqcup_{i=1}^mO_{ij}\times N_i=O_j.\]
This verifies that under the action $\alpha$, one has $F\prec O_j$ for $j=1,2$ as desired.
\end{proof}

\begin{proposition}
Suppose that $\alpha: G\curvearrowright X\times Y$ is the action in Proposition 5.1. Then $\alpha$ has the uniform tower property.
\end{proposition}
\begin{proof}
Let $O, U$ be open subsets of $X\times Y$ such that $\overline{O}\subset U$. Let $T$ be a finite subset of $G$.
For every $(x,y)\in \overline{O}$ there is an open neighbourhood $M_x\times N_y$ of $x$ such that $(x,y)\in M_x\times N_y\subset U$.  All of these neighbourhoods
form an open cover of $\overline{O}$ so that we can choose a finite subcover, say, $\overline{O}\subset\bigcup_{i=1}^n M_i\times N_i\subset U$.

Now, since the action $\beta: G\curvearrowright X$ is topologically free, $D_T=\{x\in X: \beta_t(x)\neq x\ \textrm{for all}\ t\in T^{-1}T\setminus\{e\}\}$ is open dense in $X$. Now, for each $i=1,2,\dots, n$ choose a point $x_i\in
 M_i\cap D_T$ and an open neighbourhood $O_i$ of $x_i$ such that $x_i\in O_i\subset M_i$ and $(T,O_i)$ form an open tower in $X$. In addition, by the following argument, $\{x_i, O_i: i=1,2,\dots, n\}$ can be chosen properly such that $TO_i\cap TO_j=\emptyset$ whenever $1\leq i\neq j\leq n$.

We can do this since the space $X$ is Hausdorff and perfect. We do this by induction until $n$. First choose $x_1\in M_1\cap D_T$. Then the points in $\{\beta_t(x_1):t\in T\}$ are
pairwise different. Suppose that for a $k<n$ the set of different points $\{\beta_t(x_i):t\in T, i=1,2,\dots, k\}$ has been defined. Then choose $x_{k+1}\in (M_{k+1}\cap D_T)\setminus \{\beta_t(x_i):t\in T^{-1}T, i=1,2,\dots, k\}$. We can do this since the space is perfect. The resulting points in $\{\beta_t(x_i): t\in T, i=1,2,\dots, n\}$ are pairwise different. Since $X$ is Hausdorff, there is a family of pairwise disjoint open sets  $\{O_{tx_i}: t\in T, i=1,2,\dots, n\}$ such that
$\beta_t(x_i)\in O_{tx_i}$ for $t\in T$ and $i=1,2,\dots, n$. Then define $O_i=\bigcap_{t\in T}\beta_{t^{-1}}(O_{tx_i})\cap M_i$ for $i=1,2,\dots, n$.

Then for each $i=1,2,\dots, n$ choose an open non-empty set $P_i$ such that $\overline{P_i}\subset O_i\subset M_i$. Now consider $\pi_Y(\overline{O})$, a compact subset of $\bigcup_{i=1}^n N_i$. By the argument of partition of unity, for each $i=1,2,\dots, n$ there is an open non-empty set $H_i$ such that $\overline{H_i}\subset N_i$ and $\pi_Y(\overline{O})\subset \bigcup_{i=1}^n \overline{H_i}$. Now define $W=\bigsqcup_{i=1}^n O_i\times N_i$ and $F=\bigsqcup_{i=1}^n \overline{P_i}\times \overline{H_i}$. Observe that $F\subset W\subset U$.

We claim $(T, W)$ is a tower. Indeed  for two distinct elements $t, s\in T$, one has
\begin{align*}
\alpha_t(W)\cap \alpha_s(W)&=(\bigsqcup_{i=1}^n \beta_t(O_i)\times N_i)\cap (\bigsqcup_{j=1}^n \beta_s(O_j)\times N_j)\\
&=\bigsqcup_{i=1}^n\bigsqcup_{j=1}^n(\beta_t(O_i)\cap \beta_s(O_j))\times (N_i\cap N_j)=\emptyset.
\end{align*}
Finally, we prove $O\cap Z\neq \emptyset$ implies that $F\cap Z\neq \emptyset$ for all $G$-invariant closed subsets $Z$ of $X$. First one has
\[\pi_Y(F)=\bigcup_{i=1}^n\overline{H_i}\supset \pi_Y(\overline{O})\supset \pi_Y(O).\]
Then let $Z$ be a $G$-invariant closed subset of $X\times Y$. Then $Z$ necessarily is of the form $X\times P$ for some closed $P\subset Y$ by Proposition 5.1. Suppose that $O\cap Z\neq \emptyset$. Then $\emptyset\neq \pi_Y(O\cap Z)\subset \pi_Y(O)\cap P$ and thus $\pi_Y(F)\cap P\neq \emptyset$. This implies that $F\cap Z=F\cap (X\times P)\neq \emptyset$ as desired.
\end{proof}

Now we are ready to prove Corollary 1.5.

\begin{proof}(Corollary 1.5)
Since the group $G$ is exact, the action $\alpha: G\curvearrowright X\times Y$ is exact. In addition,  Proposition 5.1, 5.2 and 5.3 show that the action $\alpha: G\curvearrowright X\times Y$ is essentially free and has paradoxical comparison as well as the uniform tower property. This means that $\alpha$ satisfies all conditions of Theorem 1.3 and thus $C(X\times Y)\rtimes_{\alpha,r} G$ is purely infinite.
\end{proof}

The following explicit example is a direct application of Corollary 1.5.

\begin{example}
In particular, for an exact group $G$ consider a
topologically free, amenable strong boundary action $\beta: G\curvearrowright X$ on a compact metrizable space $X$. Such an example exists, like Example 2.2 in \cite{L-S}. Let $Y$ be a compact Hausdorff space and $\alpha: G\curvearrowright X\times Y$ be the action mentioned above. Then $C(X\times Y)\rtimes_{\alpha,r} G$ is purely infinite.
\end{example}

\section{The type semigroup $W(X,G)$}
Throught this section $X$ denotes the Cantor set. We will study the type semigroup associated to an action $\alpha: G\curvearrowright X$. To begin the story, we recall some general background information.

A state on a preordered monoid $(W, +, \leq)$ is an order preserving morphism $f: W\rightarrow [0,\infty]$.  We say that a state is \textit{non-trivial} if it takes a value different from $0$ and $\infty$. We denote by $S(W)$ the set consisting of all states of $W$ and by $SN(W)$ the set of all non-trivial states. We write $S(W, x)=\{f\in S(W): f(x)=1\}$, which is a subset of $SN(W)$.

We say that an element $x\in W$ is \textit{properly infinite} if $2x\leq x$. We say that the monoid $W$  is \textit{purely infinite} if every $x\in W$ is properly infinite. In addition, we say that the monoid $W$ is \textit{almost unperforated} if, whenever $x,y\in W$ and $n\in \mathbb{N}$ are such that $(n+1)x\leq ny$, one has $x\leq y$. The following proposition due to Ortega, Perera, and R{\o}rdam is very useful.

\begin{proposition}(\cite[Proposition 2.1]{O-P-R})
Let $(W,+,\leq)$ be an ordered abelian semigroup, and let $x,y\in W$. Then the following conditions are equivalent:
\begin{enumerate}[label=(\roman*)]
\item There exists $k\in \mathbb{N}$ such that $(k+1)x\leq ky$.

\item There exists $k_0\in \mathbb{N}$ such that $(k+1)x\leq ky$ for every $k\geq k_0$.

\item There exists $m\in \mathbb{N}$ such that $x\leq my$ and $f(x)<f(y)$ for every state $f\in S(W,y)$.
\end{enumerate}

\end{proposition}

For an action $\alpha: G\curvearrowright X$, we can associate to it a preordered
monoid called the \textit{type semigroup} (see \cite{D}, \cite{R-S}, \cite{TimR} and \cite{Wagon}) .We will use the following formulation that appears in \cite{D} and \cite{TimR}.  We again write $\alpha$ for the induced action  on $C(X)$, which is given by $\alpha_s(f)(x)=f(s^{-1}x)$ for all $s\in G$, $f\in
C(X)$, and $x\in X$. On the space $C(X, \mathbb{Z}_{\geq 0})$ consider the equivalence relation defined by
$f\sim g$ if there are $h_1, h_2,\dots h_n\in C(X,  \mathbb{Z}_{\geq 0})$ and $s_1, s_2, \dots, s_n\in G$
such that $\sum_{i=1}^n h_i=f$ and $\sum_{i=1}^n \alpha_{s_i}(h_i)=g$. We write $W(X, G)$ for the quotient $C(X,
\mathbb{Z}_{\geq 0})/\sim$ and define an operation on $W(X, G)$ by $[f]+[g]=[f+g]$.  Moreover, we endow
$W(X, G)$ with the algebraic  order, i.e., for $a,b\in W(X, G)$ we declare that $a\leq b$ whenever there exists a
$c\in W(X, G)$ such that $a+c=b$. Then it can be verified that $W(X, G)$ is a well-defined preordered Abelian
semigroup. We call it  the \textit{type semigroup} of $\alpha$.

In this Cantor set context, we can rephrase the dynamical comparison in the language of the type semigroup. In fact Proposition 3.5 in \cite{D} implies that for all clopen subsets $A, B$ of $X$ one has $A\prec B$ if and only if $[1_A]\leq [1_B]$. In addition, if there is no $G$-invariant Borel probability measure on $X$, Proposition 3.6 in \cite{D} shows that the action has dynamical comparison if and only if $[1_A]\leq [1_B]$ for all clopen subsets $A, B$ of $X$.

We remark that $SN(W(X, G))=\emptyset$ if the action is minimal and there is no $G$-invariant Borel probability measure. Indeed, Lemma 5.1 in \cite{R-S} shows that if the action is minimal then every state in  $SN(W(X,G))$ induces a non-trivial Borel probability measure on $X$. Therefore $M_G(X)=\emptyset$ implies that $SN(W(X, G))=\emptyset$ provided the action is minimal.

The proof of the following proposition contains ideas from Lemma 13.1 in \cite{D}.

\begin{proposition}
Let $\alpha: G\curvearrowright X$ be a minimal action such that there is no $G$-invariant Borel probability measure on $X$. Then $W(X, G)$ is almost unperforated if and only if $\alpha$ has dynamical comparison.
\end{proposition}
\begin{proof}
Suppose that $W(X, G)$ is almost unperforated. To show that $\alpha$ has dynamical comparison, by the discussion above it suffices to show that for all clopen subsets $A, B\subset X$ we have $[1_A]\leq [1_B]$. Since the action $\alpha$ is minimal, $X$ is
covered by finitely many translates of $B$. This implies that $[1_A]\leq [1_X]\leq m[1_B]$ for some $m\in \mathbb{N}$.
Observe that $S(W(X,G), [1_B])\subset SN(W(X, G))=\emptyset$ by the remark above. It follows from Proposition 6.1 that there exists an $n\in \mathbb{N}$ such that $(n+1)[1_A]\leq
n[1_B]$. Then the almost unperforation of $W(X, G)$ entails that $[1_A]\leq [1_B]$ as desired.

For the converse direction, we show that if $\alpha$ has dynamical comparison then $[f]\leq [g]$ for all non-zero $[f],[g]\in W(X, G)$, which trivially implies almost unperforation. First, since $\alpha$ has dynamical comparison then for all clopen subsets $A, B$ of $X$ one has $[1_A]\leq [1_B]$.  Let $f, g\in C(X, \mathbb{Z}_{\geq 0})$, we
can write $f=\sum_{i=1}^n 1_{A_i}$  and $g=\sum_{j=1}^m 1_{B_j}$, where $A_i=\{x\in X: f(x)\geq i\}$ and
$B_j=\{x\in X: g(x)\geq j\}$ with $n=\max_{x\in X}f(x)$ and $m=\max_{x\in X}g(x)$. Since $[1_{A_i}]\leq [1_{B_i}]$ for every $i\leq
n$, if $n\leq m$, we have
\begin{align*}
[f]=\sum_{i=1}^n[1_{A_i}]\leq \sum_{i=1}^n[1_{B_i}]\leq [g]
\end{align*}
Suppose that $n>m$.  Choose $n-m+1$ many pairwise disjoint nonempty clopen subsets of $B_m$, denoted by $\{C_k: k=0,1,\dots,
n-m\}$. Then dynamical comparison implies that $[1_{A_{m+k}}]\leq [1_{C_k}]$ for $k=0, 1,\dots, n-m$. Now we have
\begin{align*}
[f]=\sum_{i=1}^{m-1}[1_{A_i}]+\sum_{k=0}^{n-m}[1_{A_{m+k}}]\leq \sum_{j=1}^{m-1}[1_{B_j}]+\sum_{k=0}^{n-m}[1_{C_k}]\leq  \sum_{j=1}^{m}[1_{B_j}]=[g]
\end{align*}

This verifies that $[f]\leq [g]$ for all $[f],[g]\in W(X, G)$.
\end{proof}

Recall that under the assumption that $G$ is amenable and $\alpha$ is minimal and free the results of Kerr \cite{D} and Kerr-Szab\'{o} \cite{D-G} show that $W(X, G)$ is almost unperforated if and only if the action $\alpha: G\curvearrowright X$ has dynamical comparison. Proposition 6.2 above, on the other hand, provides us the same conclusion in the case that $M_G(X)=\emptyset$. Recall that if an action $\alpha: G\curvearrowright X$ is amenable then we have a dichotomy that either $G$ is amenable or $M_G(X)$ is the empty set. Therefore, by combining Proposition 6.2 with the result due to Kerr and Kerr-Szab\'{o}, we have the following result as a Cantor dynamical analogue of R{\o}rdam's celebrated result on the equivalence between the strict comparison and the almost unperforation of Cuntz semigroup for unital simple nuclear $C^*$-algebras (see \cite{Ror2}).

\begin{corollary}
Let $\alpha: G\curvearrowright X$ be an amenable minimal free action. Then $W(X, G)$ is almost unperforated if and only if $\alpha$ has dynamical comparison.
\end{corollary}

In veiw of this result, it is natural to investigate whether we have a similar result if we drop the conditions of minimality and freeness of the action. In the rest of this section, we address this problem  in the case that there is no $G$-invariant non-trivial measures.  To this end, we prove Theorem 1.6 to establish the equivalence between paradoxical comparison and the almost unperforation of the type semigroup of a non-necessarily minimal action on the Cantor set. Recall that we have shown that paradoxical comparison on the Cantor set implies that there is no $G$-invariant non-trivial Borel measure. Then the answer hides in the following theorem, which is a slightly stronger version of Theorem 5.4 in \cite{R-S}. This theorem shows the relationship among the type semigroup, $C^\ast$-algebras and paradoxical comparison. We need to say that we add no new condition at all to Theorem 5.4 in \cite{R-S} since  our paradoxical comparison is equivalent to the condition that every clopen subset of $X$ is $(G, \tau_X)$-paradoxical on the Cantor set. However, what is new here is the equivalence of (i), (ii) and (iii) without the hypothesis of almost unperforation.

\begin{theorem}
Let $\alpha: G\curvearrowright X$ be an continuous action with $G$ exact and $X$ the Cantor set. Suppose that the action $\alpha$ is essentially free. Consider the following properties.
\begin{enumerate}[label=(\roman*)]

\item $\alpha$ has paradoxical comparison;

\item $W(X, G)$ is purely infinite;

\item Every clopen subset of $X$ is $(G, \tau_X)$-paradoxical;

\item The $C^\ast$-algebra $C(X)\rtimes_r G$ is purely infinite;

\item The $C^\ast$-algebra $C(X)\rtimes_r G$ is traceless in the sense that it admits no non-zero lower semi-continuous (possibly unbounded) 2-quasitraces defined on an algebraic ideal (see \cite{K-Ror});

\item There are no non-trivial states on $W(X, G)$.
\end{enumerate}
Then (i)$\Leftrightarrow$(ii)$\Leftrightarrow$(iii)$\Rightarrow$(iv)$\Rightarrow$(v)$\Rightarrow$(vi). Moreover, if $W(X, G)$ is almost unperforated then (vi)$\Rightarrow$(i), whence all of these properties are equivalent.
\end{theorem}
\begin{proof}
It has been proved in \cite{R-S} that (ii)$\Rightarrow$(iii)$\Rightarrow$(iv)$\Rightarrow$(v)$\Rightarrow$(vi) and (vi)$\Rightarrow$(ii) whenever $W(X, G)$ is almost unperforated. We have verified (iii)$\Leftrightarrow$(i) in general in the paragraph after Definition 4.3. Therefore it suffices to show (i)$\Rightarrow$(ii).

(i)$\Rightarrow$(ii). Fix an element $[f]\in W(X, G)$. Write one of its representative $f$  to be $f=\sum_{i=1}^n 1_{A_i}$, where $A_i=\{x\in X: f(x)\geq i\}$ with $n=\max_{x\in X}f(x)$. Since $\alpha$ has paradoxical comparison, for each $A_i$ one finds two disjoint open subsets $U_{i,1}$ and $U_{i,2}$ of $A_i$ such that $A_i\prec U_{i,1}$ and $A_i\prec U_{i,2}$. Then for $j=1,2$, Proposition 3.5 in \cite{D} allows us to find a finite clopen partition $\mathcal{P}^{(j)}=\{V_1^{(j)}, \dots, V_{n_j}^{(j)}\}$ of $A_i$ and group elements $s_1^{(j)},\dots, s_{n_j}^{(j)}\in G$ such that $\bigsqcup_{k=1}^{n_j}s_k^{(j)}V_k^{(j)}\subset U_{i,j}$. We may assume each $U_{i,j}$ is clopen by redefining $U_{i,j}\defeq\bigsqcup_{k=1}^{n_j}s_k^{(j)}V_k^{(j)}$ for each $j=1,2$. This implies that $[1_{A_i}]\leq [1_{U_{i,j}}]$ for $j=1,2$. This implies that
\begin{align*}
[1_{A_i}]+[1_{A_i}]\leq [1_{U_{i,1}}]+[1_{U_{i,2}}]\leq [1_{A_i}].
\end{align*}
Therefore we have $2[f]=2[\sum_{i=1}^n 1_{A_i}]\leq [\sum_{i=1}^n 1_{A_i}]=[f]$, which means that $W(X, G)$ is purely infinite.

\end{proof}

Now we are ready to prove Theorem 1.6.

\begin{proof}(Theorem 1.6)
Recall that Lemma 5.1 in \cite{R-S} shows that every non-trivial state on $W(X, G)$ induces a non-trivial $G$-invariant Borel measure. Then from the assumption that there is no non-trivial Borel measure one has that $SN(W(X, G))=\emptyset$.

Now suppose that the type semigroup $W(X, G)$ is almost unperforated. The proof of (v)$\Rightarrow$(i) of Theorem 5.4 in \cite{R-S} (i.e. (vi)$\Rightarrow$(ii) in our Theorem 6.4) implies that $W(X, G)$ is purely infinite and thus the action $\alpha$ has paradoxical comparison by Theorem 6.4. We remark that the proof of this implication  does not require the action to be essentially free.

For the converse direction, suppose that $\alpha$ has paradoxical comparison. We have shown in the proof of Theorem 6.4 that the type semigroup $W(X, G)$ is purely infinite, which means $2[f]\leq [f]$ for every $[f]\in W(X, G)$. By induction we have $m[f]\leq [f]$ for every $m\in \mathbb{N}$. Now suppose that $(n+1)[g]\leq n[f]$ for some $n\in \mathbb{N}$ and $[f], [g]\in W(X,G)$. We have $[g]\leq (n+1)[g]\leq n[f]\leq [f]$. This shows that the type semigroup is almost unperforated.
\end{proof}

\section{Acknowledgement}
The author should like to thank his supervisor David Kerr for an inspiring suggestion leading to a much more  general definition of paradoxical comparison than  the  author's original one as well as his helpful comments and corrections. He would like to thank Kang Li and Jianchao Wu for their valuable discussions and comments. He also thanks Jamali Saeid for letting him know the reference \cite{L-S} for this topic. Finally, he should like to thank the anonymous referee whose comments and suggestions helped a lot to improve the paper.

\end{document}